\documentclass[oneside]{amsart}

\usepackage{amsthm}
\usepackage{amsmath, amssymb}
\usepackage[all,dvips,arc,curve,frame]{xy}
\usepackage{enumerate}
\usepackage[T1]{fontenc}
\usepackage{pslatex} 

\theoremstyle{plain}

\newtheorem{thm}{Theorem}
\newtheorem{cor}[thm]{Corollary}
\newtheorem{pro}[thm]{Proposition}

\newtheorem{lem}[thm]{Lemma}

\theoremstyle{definition}

\newtheorem{ques}[thm]{Question}

\newtheorem{rem}[thm]{Remark}

\newcommand{\mygraph}[1]{\xybox{\xygraph{#1}}}

\def\C{\mathbb{C}}

\def\Bir{{\text{\rm Bir}}}

\def\cpo{{\mathbb{P}^1}}

\def\cpt{{\mathbb{P}^3}}
\def\cpq{{\mathbb{P}^4}}

\renewcommand{\phi}{\varphi}

\usepackage[colorlinks, linktocpage, citecolor = blue, linkcolor = blue]{hyperref}

\begin{document}

\title{On the genus of birational maps between 3-folds}
\author{St\'ephane Lamy}
\address{Institut de Math\'ematiques de Toulouse, Universit\'e Paul Sabatier, 
118 route de Narbonne, 31062 Toulouse Cedex 9, France}
\email{slamy@math.univ-toulouse.fr}
\thanks{I thank J. Blanc and A. Dubouloz who pointed out and helped me fix some silly mistakes in an early version of this work.
This research was supported by ANR Grant "BirPol"  ANR-11-JS01-004-01.}

\maketitle

In this note we present two equivalent definitions for the genus of a birational map $\phi\colon X \dashrightarrow Y$ between smooth complex projective 3-folds.
The first one is the definition introduced by M. A. Frumkin in \cite{F}, the second one was recently suggested to me by S. Cantat.
By focusing first on proving that these two definitions are equivalent, one can obtain all the results in \cite{F} in a much shorter way.
In particular, the genus of an automorphism of $\C^3$, view as a birational self-map of $\cpt$,  will easily be proved to be 0.

\section{Preliminaries}

By a \textbf{$n$-fold} we always mean a smooth projective variety  of dimension $n$ over $\C$. 

Let  $\phi\colon X \dashrightarrow Y$ be a birational map between $n$-folds.
We assume that a projective embedding of $Y$ is fixed once and for all, hence $\phi$ corresponds to the linear system on $X$ given by preimages by $\phi$ of hyperplane sections on $Y$. 

We call \textbf{base locus} of $\phi$ the base locus of the linear system associated with $\phi$: this is a subvariety of codimension at least 2  of $X$, which corresponds to the indeterminacy set of the map. 

Another subvariety of $X$ associated with $\phi$ is the \textbf{exceptional set}, which is defined as the complement of the maximal open subset where $\phi$ is a local isomorphism.
If $X = Y = \mathbb{P}^n$ the exceptional set (given by the single equation Jacobian $=0$) has pure codimension 1, but this not the case in general: consider for instance the case of a flop, or more generally of any isomorphism in codimension 1, where the exceptional set coincides with the base locus.     

A \textbf{regular resolution} of $\phi$ is a morphism $\sigma\colon Z \to X$ which is a sequence of blow-ups $\sigma = \sigma_1 \circ \dots \circ \sigma_r$ along smooth irreducible centers, such that $Z \to Y$ is a birational morphism, and such that each center $B_i$ of the blow-up $\sigma_i\colon Z_i \to Z_{i-1}$ is contained in the base locus of the induced map $Z_{i-1} \dashrightarrow Y$.
Recall that as a standard consequence of resolution of singularities, a regular resolution always exists. 

$$\mygraph{
!{<0cm,0cm>;<1cm,0cm>:<0cm,1cm>::}
!{(0,0)}*++{Z_0 = X}="X"
!{(1,1)}*+{Z_1}="Z1"
!{(2.5,2.5)}*++{Z_r = Z}="Z"
!{(3.5,0)}*++{Y}="Y"
"X"-@{-->}_\phi"Y" "Z"-@{.>}"Z1"-@{->}^>(.6){\sigma_1}"X"  "Z"-@{->}"Y" 
 "Z"-@{->}@/_5mm/_\sigma"X" 
}$$

We shall use the following basic observations about the exceptional set and the base locus  of a birational map.

\begin{lem} \label{lem:E}
\begin{enumerate}
\item Let $\tau \colon X \to Y$ be a birational morphism between 3-folds. Then through a general point of any component of the exceptional set of $\tau$, there exists a rational curve contracted by $\tau$.
\item Let $\phi \colon X \dashrightarrow Y$ be a birational map between 3-folds, and let $E \subset X$ be an irreducible divisor contracted by $\phi$. 
Then $E$ is birational to $\cpo \times C$ for a smooth curve $C$. 
\end{enumerate}
\end{lem}

\begin{proof}
For the first assertion (which is in fact true in any dimension), see for instance \cite[Proposition 1.43]{Deb}.
When $\phi$ is a morphism, the second assertion is in fact what is first proved in \cite{Deb}. 
Finally, when $\phi$ is not a morphism, we reduce to the previous case by considering a resolution of $\phi$.
\end{proof}

\begin{lem}\label{lem:basic}
Let $\phi \colon X \dashrightarrow Y$ be a birational map between $n$-folds, and consider 
$$\xymatrix{
& Z \ar[ld]_{\sigma} \ar[dr]^{\tau} \\
X \ar@{-->}[rr]_{\phi} && Y
}$$
a regular resolution of $\phi$.
Then a point $p \in X$ is in the base locus of $\phi$ if and only if the set $\tau(\sigma^{-1}(p))$ has dimension at least 1.
\end{lem}

\begin{proof}
If $p$ is not in the base locus of $\phi$ then by regularity of the resolution $\sigma^{-1}(p)$ is a single point, and thus $\tau(\sigma^{-1}(p))$ as well.

Now suppose that $p$ is in the base locus of $\phi$, and consider $H_Y$ a general hyperplane section of $Y$. 
Denote by $H_X$, $H_Z$ the strict transform of $H_Y$ on $X$ and $Z$ respectively.
By definition of the base locus, we have $p \in H_X$, hence 
$$\sigma^{-1}(p) \cap H_Z \neq \emptyset \quad\text{ and }\quad \tau(\sigma^{-1}(p)) \cap H_Y \neq \emptyset.$$ 
This implies that $\tau(\sigma^{-1}(p))$ has positive dimension.
\end{proof}

We will consider blow-ups of smooth irreducible centers in 3-folds.
If $B$ is such a center, $B$ is either a point or a smooth curve.
We define the \textbf{genus} $g(B)$ to be 0 if $B$ is a point, and the usual genus if $B$ is a curve.
Similarly, if $E$ is an irreducible divisor contracted by a birational map between 3-folds, then by Lemma \ref{lem:E} $E$ is birational to a product $\cpo \times C$ where $C$ is a smooth curve, and we define the \textbf{genus} $g(E)$ of the contracted divisor to be the genus of $C$.

\section{The two definitions}

Consider now a birational map $\phi\colon X \dashrightarrow Y$ between 3-folds, and let $\sigma\colon Z \to Y$ be a regular resolution of $\phi^{-1}$.

Frumkin \cite{F} defines the genus $g(\phi)$ of $\phi$ to be the maximum of the genus among the centers of the blow-ups in the resolution $\sigma$. Remark that this definition depends a priori from a choice of regular resolution, and Frumkin spends a few pages in order to show that in fact it does not.

During the social dinner of the conference \textit{Groups of Automorphisms in Birational and Affine Geometry},
S. Cantat suggested to me another definition, which is certainly easier to handle in practice: define the genus of $\phi$ to be the maximum of the genus among the irreducible divisors in $X$ contracted by $\phi$.

Denote by $F_1,\dots,F_r$ the exceptional divisors of the sequence of blow-ups $\sigma = \sigma_1 \circ \dots \circ \sigma_r$, or more precisely their strict transforms on $Z$.
On the other hand, denote by $E_1, \dots, E_s$ the strict transforms on $Z$ of the irreducible divisors contracted by $\phi$.

Note that if $\phi^{-1}$ is a morphism, then both collections $\{F_i\}$ and  $\{E_i\}$ are empty. 
In this case, by convention we say that the genus of $\phi$ is 0.
In this section we prove:

\begin{pro} \label{pro:samedef}
Assume $\phi^{-1}$ is not a morphism. Then 
$$\max_{i = 1,\dots,s} g(E_i) = \max_{i = 1,\dots,r} g(F_i).$$
In other words the definition of the genus by Frumkin coincides with the one suggested by Cantat, and in particular it does not depend on a choice of a regular resolution. 
\end{pro}

Denote by $B_i$ the center of the blow-up $\sigma_i$ producing $F_i$.
We define a \textbf{ partial order} on the divisors $F_i$ by saying that $F_j \succcurlyeq F_k$ if one of the following conditions is verified:
\begin{enumerate}[$(i)$]
\item $j = k$;
\item $j > k$, $B_k$ is a point, and $B_j$ is contained in the strict transform of $F_k$;
\item $j > k$, $B_k$ is a curve,  and $B_j$ intersects the general fiber of the strict transform of the ruled surface $F_k$.
\end{enumerate}
We say that $F_i$ is \textbf{essential} if $F_i$ is  a maximal element for the order $\succcurlyeq$.

\begin{lem} \label{lem:ess}
The maximum $\max_i g(F_i)$ is realized by an essential divisor.
\end{lem}

\begin{proof}
We can assume that the maximum is not $0$, otherwise there is nothing to prove.
Consider $F_k$ realizing the maximum, and $F_j \succcurlyeq F_k$ with $j > k$. 
Then the centers $B_j, B_k$ of $\sigma_j$ and $\sigma_k$ are curves, and $B_j$ dominates $B_k$ by a morphism.
By the Riemann-Hurwitz formula, we get $g(F_j) \ge g(F_k)$, and the claim follows.
\end{proof}

\begin{lem} \label{lem:essinEi}
The subset of the essential divisors $F_i$ with $g(F_i) \ge 1$ is contained in the set of the contracted divisors $E_j$.
\end{lem}

\begin{proof}
Let $B_i \subset Z_{i-1}$ be the center of a blow-up producing a non-rational essential divisor $F_i$, and consider the diagram:
$$\xymatrix{
& Z \ar[ld]_{\tau} \ar[dr]^{\tilde\sigma = \sigma_i \circ \dots \circ  \sigma_r} \\
X  && Z_{i-1}\ar@{-->}[ll]_{\psi_{i-1}}
}$$
By applying Lemma \ref{lem:basic} to $\psi_{i-1}\colon Z_{i-1} \dashrightarrow X$, we get $\dim \tau(\tilde\sigma^{-1}(p)) \ge 1$ for any point $p \in B_i$.
Since $F_i$ is essential, $l_p := \tilde\sigma^{-1}(p)$ is a smooth rational curve contained in $F_i$ for all except finitely many $p \in B_i$. 
So $\tau(l_p)$ is also a curve. 
If $\tau(l_p)$ varies with $p$, then $\tau(F_i)$ is a divisor, which is one of the $E_i$. 
Now suppose $\tau(l_p)$ is a curve independent of $p$, that means that $F_i$ is contracted to this curve by $\tau$.
Consider $q$ a general point of $F_i$. 
By Lemma \ref{lem:E} there is a rational curve $C \subset F_i$ passing through $q$ and contracted by $\tau$, but this curve should dominate the curve $B_i$ of genus $\ge 1$: contradiction. 
\end{proof}

\begin{proof}[Proof of Proposition \ref{pro:samedef}]
Observe that the strict transform of a divisor contracted by $\phi$ must be contracted by $\sigma$, hence we have the inclusion $\{E_i\} \subset \{F_i\}$. This implies $\max_{i} g(E_i) \le \max_i g(F_i)$.
If all $F_i$ are rational, then the equality is obvious.

Suppose at least one of the $F_i$ is non-rational. 
By Lemma \ref{lem:essinEi} we have the inclusions
$$ \{F_i; F_i \text{ is non-rational and essential}\} \subset \{E_i\} \subset \{F_i\}.$$
Taking maximums, this yields the inequalities
$$\max_i \{g(F_i); F_i \text{ is non-rational and essential}\} \le \max_{i} g(E_i) \le \max_i g(F_i).$$ By Lemma \ref{lem:ess} we conclude that these three maximums are equal.
\end{proof}

\section{Some consequences}

The initial motivation for a reworking of the paper of Frumkin was to get a simple proof of the fact that a birational self-map of $\cpt$ coming from an automorphisms of $\C^3$ admits a resolution by blowing-up points and rational curves:
  
\begin{cor} \label{cor:C3&psauto}
The genus of $\phi$ is zero in the following two situations:
\begin{enumerate}
\item  $\phi \in \Bir(\cpt)$ is the completion of an automorphism of $\C^3$;
\item $\phi\colon X \dashrightarrow Y$ is a pseudo-isomorphism (i.e. an isomorphism in codimension 1).
\end{enumerate}
In particular for such a map $\phi$ any regular resolution only involves  blow-ups of points and of smooth rational curves. 
\end{cor}

\begin{proof}
Both results are obvious using the definition via contracted divisors!
\end{proof}

I mention the following result for the sake of completeness, even if I essentially follow the proof of Frumkin (with some slight simplifications).

\begin{pro}[compare with {\cite[Proposition 2.2]{F}}] \label{2.2}
Let $\phi\colon X \dashrightarrow Y$ be a birational map between 3-folds, and let $\sigma\colon Z \to X$, $\sigma'\colon Z' \to Y$ be resolutions of $\phi$, $\phi^{-1}$ respectively.
Assume that $\sigma$ is a regular resolution, and denote by $h\colon Z \dashrightarrow Z'$ the induced birational map.
Then $g(h) = 0$. 
$$
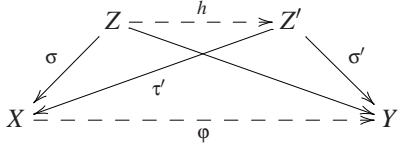
\xymatrix{
& Z \ar@{-->}[rr]^h \ar[ld]_\sigma \ar[drrr] && Z' \ar[llld]^{\tau'} \ar[dr]^{\sigma'} \\
X \ar@{-->}[rrrr]_\phi &&&& Y
}$$ 
\end{pro}

\begin{proof}
We write $\sigma = \sigma_1 \circ \dots \circ \sigma_r$, where $\sigma_i\colon Z_i \to Z_{i-1}$ is the blow-up of a smooth center $B_i$. 
Note that $Z_0 = X$, and $Z_r = Z$.
Assume that the induced map $h_i\colon Z_i \dashrightarrow Z'$ has genus 0 (this is clearly the case for $h_0 = {\tau'}^{-1}$), and let us prove the same for $h_{i+1} = h_i \circ \sigma_{i+1}$.
We can assume that $\sigma_{i+1}$ is the blow-up of a non-rational smooth curve $B_{i+1}$, otherwise there is nothing to prove.
Consider 
$$
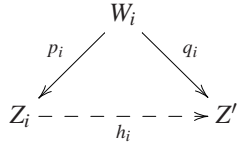
\xymatrix{
& W_i \ar[ld]_{p_i} \ar[dr]^{q_i}\\
Z_i \ar@{-->}[rr]_{h_i} && Z'
}$$
a regular resolution of $h_i^{-1}$.
Since $W_i$ dominates $Y$ via the morphism $\sigma' \circ q_i$, the curve $B_{i+1}$ is in the base locus of $p_i^{-1}$: otherwise $B_{i+1}$ would not be in the base locus of $Z_i \dashrightarrow Y$, contradicting the regularity of the resolution $\sigma$.
Thus for any point $x \in B_{i+1}$, $p_i^{-1}(x)$ is a curve, and there exists an open set $U \subset Z_i$ such that $p_i^{-1} (U \cap B_{i+1})$ is a non-empty divisor. 
By applying over $U$ the universal property of blow-up  (see \cite[Proposition II.7.14]{Har}), we get that there exists an irreducible divisor on $W_i$ whose strict transform on $Z_{i+1}$ is the exceptional divisor $E_{i+1}$ of $\sigma_{i+1}$.
Hence the birational map $p_i^{-1}\circ\sigma_{i+1}\colon Z_{i+1} \dashrightarrow W_i$ does not contract any divisor and so has genus 0.
Composing by $q_i$ which also has genus 0 by hypothesis we obtain $g(h_{i+1}) = 0$.
By induction we obtain $g(h_r) = 0$, hence the result since $h_r = h$.
\end{proof}

\begin{rem}
In the setting of Proposition \ref{2.2}, even if $\sigma'$ is also a regular resolution, there is no reason for $h\colon Z \dashrightarrow Z'$ to be a pseudo-isomorphism.
For instance, if $\phi$ admits a curve $C$ in its base locus, one could construct regular resolutions of $\phi$ starting blowing-up arbitrary many points on $C$, and so the Picard number of $Z$ can be arbitrary large (thanks to the referee who pointed this fact out).  

On the other hand, it is not clear if we could restrict the definition of a regular resolution (for instance allowing the blow-up of a point only if it is a singular point of the base locus), such that the regular resolutions $\sigma$ and $\sigma'$ would lead to $3$-folds $Z$ and $Z'$ with the same Picard numbers, and with $h$ a pseudo-isomorphism.
In such a case, Proposition \ref{2.2} would follow from Corollary \ref{cor:C3&psauto}.
\end{rem}
 
The next result is less elementary.

\begin{pro} \label{pro:inverse}
Let $X$ be a 3-fold with Hodge numbers $h^{0,1} = h^{0,3} = 0$,  and let $\phi\colon X \dashrightarrow X$ be a birational self-map.
Then $g(\phi) = g(\phi^{-1})$.
\end{pro}

For the proof, which relies on the use of intermediate Jacobians,
I refer to the original paper of Frumkin \cite[Proposition 2.6]{F}, or to \cite{LS} where it is proved that the exceptional loci of $\phi$ and $\phi^{-1}$ are birational (and even more piecewise isomorphic).
Note that Frumkin does not mention any restriction on the Hodge numbers of $X$, but it seems implicit since the proof  uses the fact, through the reference to \cite[3.23]{CG}, that the complex torus $\mathcal{J}(X)$ is a principally polarized abelian variety.

\begin{cor}
Let $g \ge 0$ be an integer. 
The set of birational self-maps of $\cpt$ of genus at most $g$ is a subgroup of $\Bir(\cpt)$.
\end{cor}

\begin{proof}
Stability under taking inverse is Proposition \ref{pro:inverse}, and stability under composition comes from the fact that any divisor contracted by $\phi \circ \phi'$ is contracted either by $\phi$ or by $\phi'$.
\end{proof}

\begin{ques}
The last corollary could be stated for any 3-fold satisfying the assumptions of Proposition \ref{pro:inverse}, but I am not aware of any relevant example. For instance, if $X \subset \cpq$ is a smooth cubic 3-fold, is there any birational self-map of $X$ with genus $g \ge 1$?
\end{ques}

\bibliographystyle{myalpha.bst}
\bibliography{biblio}

\end{document}